\documentclass[11pt]{amsart}
\usepackage{mathrsfs, bbm}
\newtheorem{theorem}{Theorem}
\newtheorem{lemma}[theorem]{Lemma}
\newtheorem{corollary}[theorem]{Corollary}
\newtheorem{proposition}[theorem]{Proposition}

\theoremstyle{definition}

\theoremstyle{definition}
\newtheorem{remark}[theorem]{Remark}

\numberwithin{equation}{section}
\numberwithin{theorem}{section}

 \def\cN{\mathcal N}
\def\cL{\mathcal L}
\def\cK{\mathcal K}

\def\cP{\mathcal P}
\def\cQ{\mathcal Q}
\def\cE{\mathcal E}
\def\cI{\mathcal I} 
\def\cC{{\mathbf C}}
\def\ep{\varepsilon}
\def\eps{\epsilon}
\def\lesssim{\stackrel{{}_<}{{}_\sim}}

\def\nc{\nabla\!\!\cdot\!}
\def\cn{\!\cdot\!\!\nabla}
\def\pa{\partial}
\def\pt{\pa_t} 
\def\px{\pa_x} 
\def\ccdot{{\hspace{-0.15mm}\cdot\hspace{-0.15mm}}}
\def\mR{{\mathbf R}}
\def\mD{{\mathbf{D}}}
\def\adm{{\mathbf X}}
\def\kernel{\mathrm{Ker} }
\def\image{\mathrm{Img}}

\def\nm{{{\mathsf{n}}}}

\def\vB{\mathbf B}
\def\vU{U}
\def\vV{\overline{U}}
\def\vu{u}
\def\vv{\overline{u}}

\def\nt{\nabla\!\!\times\!\!}

\def\vUn{U^P}
\def\vUa{U^Q}

\def\Va{V^\alpha}
\def\de{\delta}

\def\ba{\begin{array}}
\def\ea{\end{array}}
\def\be{\begin{equation}}
\def\ee{\end{equation}}
\def\bal{\begin{align}}
\def\eal{\end{align}}
\def\bals{\begin{align*}}
\def\eals{\end{align*}}
\def\beb{\begin{exampleblock}}
\def\eeb{\end{exampleblock}}
\def\bbk{\begin{block}}
\def\ebk{\end{block}}
\def\bpm{\begin{pmatrix}}
\def\epm{\end{pmatrix}}
\def\bma{\begin{pmatrix}}
\def\ema{\end{pmatrix}}
\def\bse{}
\def\ese{}
\def\bc{\bigr|_{\pa\Omega}}

\newcommand{\mmbox}[1]{\quad\mbox{#1}\quad}
\newcommand{\en}[1]{|\!|\!|#1|\!|\!|}

\begin{document}
\title[low-Mach-number  Euler equations]{low-Mach-number Euler equations with solid-wall boundary condition and general initial data}

\author[Bin Cheng]{Bin Cheng}
 \address{\newline
        Department of Mathematics,
    University of Michigan\newline
    530 Church St.
    Ann Arbor, MI 48109 USA}
\email[Bin Cheng]{bincheng@umich.edu}
 
\date{June 06, 2010}
\keywords{Compressible Euler equations; singular limit; initial-boundary value problem; low-Mach-number; ill-prepared initial data.}
\subjclass{35Q31 (Primary) 35L50, 76N15 (Secondary)}

\begin{abstract}We prove that the divergence-free component of the compressible Euler equations with solid-wall boundary condition converges strongly towards the incompressible Euler equations at the same order as the Mach number. General initial data   are considered and are not necessarily close to the divergence-free state. Thus,  large amplitude of fast oscillations persist  and interact through nonlinear coupling without any dissipative or dispersive machanism. It is then shown, however, that the contribution from fast oscillations to the slow dynamics through nonlinear coupling is of the same order as the Mach number when \emph{averaged in time}. The structural condition of a vorticity equation plays a key role in our argument.  
\end{abstract}  
\maketitle
\section{Introduction}
Hyperbolic Partial Differential Equations (PDEs) of multiscale nature have seen rapidly growing applications in recent years. We conduct a theoretical investigation in this paper concerning a prototypical example: the compressible Euler equations of barotropic fluids in a bounded domain in the low-Mach-number regime. The system of equations, in terms of \emph{total} density $\hat{\rho}$ and velocity $u$, read\[\left\{\begin{split}\pt\hat{\rho} +\nc(\hat{\rho}\vu)&=0,\\
\hat{\rho} (\pt\vu +\vu\cn\vu)+{\nabla p(\hat{\rho})\over\ep^2}&=0,\end{split}\right.\]
where Mach number $\ep\ll1$ brings in fast scale of oscillations. The equation of state is  \be\label{eq:state}\mbox{pressure=${p(\hat{\rho})}$ \;\;with\;\;  $p(\cdot)\in C^\infty,\;p(1)=0,\;p'(1)=1$.
}\ee
The spatial domain $\Omega\in \mR^{\mD}$ ($\mD=2$ or 3) is bounded with smooth boundary $\pa\Omega$. It is connected but not necessarily simply connected.  The solid-wall boundary condition is prescribed as\[u\ccdot\nm\bigr|_{\pa\Omega}=0.\]
Here and below, $\nm=\nm(x)$ denotes the outward normal at $x\in\pa\Omega$. Without loss of generality, impose \[{1\over|\Omega|}\int_\Omega\hat{\rho}\,dx=1\]since $\int_\Omega\hat{\rho}\,dx$ is always conserved by the dynamics.

Having  $\hat{\rho}\approx1$ as the total density, define the density perturbation\be\label{def:rho}\rho:={\hat{\rho}-1\over\ep},\ee
and rewrite the above system for the unknown pair $(\rho,\,u)$,
\be\label{Euler}\left\{\begin{split}\pt \rho +\nc(\rho\vu)+{\nc\vu\over{ \ep}}&=0\\
\pt\vu+\vu\cn\vu+{p'(1+\ep\rho)\over1+\ep\rho}\,{\nabla \rho\over{ \ep}}&=0\end{split}\right.\ee\be\label{Euler:BC}u\ccdot\nm\bigr|_{\pa\Omega}=0\;\;,\qquad\int_\Omega\rho\,dx=0.\ee

\begin{theorem}\label{thm:main}Consider the initial-boundary value problem of the $\mD$-dimensional compressible Euler equations \eqref{Euler}, \eqref{Euler:BC} subject to initial data $(\rho_0,\,u_0)\in H^m(\Omega)$ with $m>{\mD\over2}+4$. Assume $(\rho_0,u_0)$ is compatible with the boundary condition $\mbox{``}\pt^ku_0\mbox{{''}}\ccdot\nm\bigr|_{\pa\Omega}=0$ for $k<m$. Here, $\mbox{``}\pt^ku_0\mbox{{''}}$ is obtained from expressing $\pt^ku$ in terms of spatial derivatives  via \eqref{Euler} and then substituting $(\rho_0,\,u_0)$ into the expression.

Then, there exist general constants $E,T,C $  only dependent on $m$, $\Omega$ and the pressure law $p(\ccdot)$  s.t.  for all $\ep\in(0,E/\|(\rho_0,\,u_0)\|_{H^m}]$,  the solution $(\rho,u)$ exists smoothly for times $t\in[0,T/\|(\rho_0,\,u_0)\|_{H^m}]$ and $\|(\rho,\,u)\|_{H^m}\le C\|(\rho_0,\,u_0)\|_{H^m}$ uniformly for such times. 

Moreover, there exists a splitting of the velocity field $u=u^P+u^Q$ with $\nc u^P=0$ and $u^Q=\nabla\psi$ for some potential function $\psi$ s.t. for general constants $c_1,c_2$ only dependent on $m,\Omega,p(\ccdot)$ and $\|(\rho_0,\,u_0)\|_{H^m}$,
\begin{itemize}\item (strong convergence of the slow compoenent ) 
\[\max_{0\le t\le T}\|u^P(t,\ccdot)-\vv(t,\ccdot)\|_{H^{m-3}(\Omega)}\le c_1\ep \]where $\vv$ solves the incompressible Euler equations 
\be\label{Euler:inc}\left\{\begin{split}\pt \vv+\vv\cn \vv+\nabla q&=0,\\\nc \vv&=0,\end{split}\right.\ee \[\vv\ccdot\nm\bc=0,\quad \vv(0,\ccdot)=u_0^P \quad;\]
 and,
\item (weak* convergence of the acoustic part) for any smooth testing function $(\rho',\,u')$ defined on $[0,T]\times\overline{\Omega}$,
\be\label{weak:s}\left|\int_0^T\!\!\int_\Omega\rho\rho'+u^Q\ccdot u'\,dxdt\right|
\le c_2{\ep}(\|(\rho',\,u')\|_{L^\infty_tL^2_x}+\|\pt(\rho',\,u')\|_{L^2_{t,x}}).\ee
 \end{itemize}\end{theorem}
\begin{proof}
For strong convergence, see Theorem \ref{thm:strong}. Some major new ideas of this article are centered around this theorem. For weak* convergence, see Corollary \ref{thm:weak}.

For uniform existence time and $H^m$ estimates, see Theorem \ref{thm:uniform}  and the remarks thereafter. Note that the mere existence of $H^m$ solution is not the concern of this article. Thus, the compatibility condition serves an assumption  in this theorem and will not appear explicitly again.
\end{proof}
There have been numerous results regarding the {\it singular limits} of compressible Euler equations and other
fluid equations  in various settings. We point to two survey papers for some comprehensive lists of references: Schochet \cite{Sch:survey} with emphases on hyperbolic PDEs and homogenization in space-time; Masmoudi \cite{Mas:survey} with emphases on viscous fluids and weak solutions. To mention only a few earliest work, we note papers by Ebin \cite{Ebin:earlier, Ebin:Euler}, Kreiss  \cite{Kreiss:multiscale}, Beir\~{a}o de Veiga  \cite{BDV}, and Klainerman and Majda \cite{Kl:Maj}.  In a closely related paper \cite{Kreiss:derivative}, Kreiss etc. applied the bounded derivative method in numerical schemes to gain control on time derivatives and thus to rid of fast gravity waves. These results, in terms of \eqref{Euler},  confirmed that compressible flows $(\rho,\,u)$ converges to $( 0,\, \vv)$ strongly at order $O(\ep)$ with $\vv$ solving \eqref{Euler:inc} \emph{provided} the initial data $(\rho_0,\,u_0)$ also converges to $(0,\,\vv_0)$ strongly at order $O(\ep)$\be\label{well_prep}\|\rho_0\|+\|u_0-\vv_0\|\lesssim \ep\mmbox{ for some div-free $\vv_0$.}\ee
Here, $\|\cdot\|$ denotes suitable {\it spatial norm}. Note that  condition \eqref{well_prep} implies that perturbation in the total density $\hat{\rho}$ vanishes at  order $O(\ep^2)$ --- consult \eqref{def:rho}.

  This family of well-prepared  initial data lead to uniform estimates on the size of $\pt(\rho,\,  u)$ at $t=0$, independent of $\ep$, by the virtue of \eqref{Euler} and therefore suppresses the so-called initial layer. Then, one obtains uniform control on the size of $\pt(\rho,\,u)$ for finite times, which allows passing of limits by the Arzel\`{a}-Ascoli Theorem. 
Well-prepared conditions on initial data were later removed for problems in the whole space (Ukai \cite{Ukai:Rn}), in an exterior domain (Isozaki \cite{Isozaki:exterior,Isozaki:wave}) and in a torus (Schochet \cite{Sch:limits}). These arguments more or less rely on use of Fourier analysis and/or dispersive nature of the underlying wave equations. 

Singular limit problems in a bounded spatial domain, on the other hand, remain much less studied. Schochet proved in \cite{Sch:Euler} the same low-Mach-number limit with solid-wall boundary condition and, again, well-prepared initial data.  A main challenge in this setting is  the presence of characteristic boundary. It is elaborated in Rauch's work \cite{Rauch:char} for linear systems that, in general, only estimates along tangential directions  are available near the boundary. We also note that there were also preceding results  in e.g. \cite{Ebin:Euler}, \cite{BDV}, all of which required well-prepared initial data. 

The main originality of our paper is to prove the low-Mach-number limit for \emph{general initial data} without requiring the well-prepared condition \eqref{well_prep}. We establish a strong convergence rate of $O(\ep)$ for the slow incompressible part and weak* convergence for the fast acoustic part. In lieu of the commonly accepted initial layer, it is proved that nonlinear resonance of fast acoustic waves does not enter the slow dynamics at all (Lemma \ref{ff:lemma}) and interaction of fast-slow dynamics vanish at order $O(\ep)$ upon integrating in time (Lemma \ref{sf:lemma}). This argument was explicitly used in Cheng \cite{BC:HYP08} to study the rapidly rotating shallow water equations with 2 fast scales in the whole space. It recently came to our knowledge that similar ideas have occasionally appeared in the literature for weak limits (cf. equations after (7) in Lions and Masmoudi \cite{Lions:Mas:CRAS}) and for problems without boundary (cf. equation (4.27), (4.28) in Schochet \cite{Sch:survey}). We also tackle the boundary condition carefully and present a clear calculation of \emph{a priori} estimates on space-time norms of the solution. Here, the vorticity equation plays a crucial role, which was  argued in e.g. Schochet \cite{Sch:general}. Throughout our analysis (not just \emph{a priori} estimates), we employ elliptic estimates for PDE systems with general boundary conditions (Agmon etc. \cite{elliptic}).

The organization of the rest of this article is as follows. In Section \ref{sec:ell}, we introduce elliptic estimates particularly for the case at hand and present a precise characterization for the projection operators associated with Helmholtz decomposition. This projection is to be used to split the solution as well as the system into ``slow'' and ``fast'' parts. Moving on to Section \ref{sec:nonlinear}, we conduct a thorough study on nonlinear interactions of fast-fast and fast-slow types, which will confirm the assertions in the previous paragraph. Solution regularity is presumed here and is then treated rigorously in the next Section \ref{sec:a:priori} regarding \emph{a priori} estimates. A mollification method is used to handle the lack of boundary regularity. The proofs are self-contained, only relying on basic Calculus. Finally,  Section \ref{sec:RSW} contains a crash course on the rapidly Rotating Shallow Water equations to which the exact same ideas apply. It is presented in a very systematic way due to an algebraic structure of duality nature.

We will repeatedly use some well-known inequalities of Sobolev norms without making references. They are all based on H\"{o}lder's inequality,  Gargliardo-Nirenberg inequality and Sobolev inequality. For the most part, it is sufficient to accept the following estimates,
\be\label{cal:ineq}\|\pa^{j_1}_xg\pa^{j_2}_xg...\pa^{j_k}_xg\|_{L^2(\Omega)}\le c\|g\|^k_{H^m(\Omega)}\ee
where $m>\mD/2+1$, $0\le j_1\le...\le j_k\le m$ and $j_1+...+j_k\le m+1$.
 
\section{Elliptic Estimates and Helmholtz Decomposition}\label{sec:ell}
Elliptic estimates and Helmholtz decomposition are both from Elliptic PDE theories, the former regarding regularity and the latter solvability. In connection with the compressible Euler equations, the singular terms in \eqref{Euler} with fast scale $1/\ep$ define an elliptic operator,
\[\cL \bpm\rho\\u\epm:=\bpm\nc u\\\nabla\rho\epm,\]
and there is another elliptic operator\[
\cK \bpm\rho\\u\epm:=\left({0\atop\nt u}\right)\mmbox{in 3D}
\mmbox{and} \bpm 0\\0\\\nt u\epm\mmbox{in 2D}\]
that essentially yields the vorticity. They satisfy,\[\cK\cL\equiv0\mmbox{ and, in 3D,}\cL\cK\equiv0,\]
which is why the singular $1/\ep$ term does not appear in the vorticity equation.

The papers of Agmon, Douglis and Nirenberg \cite{elliptic} establish a Complementing Boundary Condition that is necessary and sufficient for the solution operator of a $s$-th order elliptic PDE system to be $\cC^m\to\cC^{m+s}$ and $H^m\to H^{m+s}$. In this article,  only a particular case is used: for any velocity field $u$ with a trace subject to the solid-wall boundary condition $u\ccdot\nm\bc=0$,
\be\label{elliptic:bc}\|u\|_{H^m(\Omega)}\le C\left(\|\nc u\|_{H^{m-1}(\Omega)}+\|\nt u\|_{H^{m-1}(\Omega)}+\|u\|_{L^2(\Omega)}\right).\ee
Here and below, we always assume $m$ is a positive integer  so that the trace $u\bc$ is well-defined. See e.g. \cite{BB:Euler} for application of this estimate. 

\begin{remark} The $\|u\|_{L^2}$ term in the above estimate is removable if and only if the elliptic PDE system\be\label{nontrivial}\nt u=0,\quad
 \nc u= 0\quad u\ccdot\nm\bc=0,\ee
admit a unique trivial solution $u\equiv0$. This is the case e.g. if  $\Omega$ is contractible (in virtue of the Poincar\'{e} lemma) but is not true in general. For example, in 2D, when $\pa\Omega$ consists at least two disjoint branches $\Gamma_1$, $\Gamma_2$, one solves a scalar Laplace equation $\Delta\psi=0$ with $\psi\bigr|_{\Gamma_i}=a_i$ and $a_1\ne a_2$. Then, $u=\bpm\pa_y\psi\\-\pa_x\psi \epm$ is indeed a nontrivial solution to the above PDE system\footnote{With a compactness argument, one can further show that these solutions form a finite dimension space. Then, \eqref{elliptic:bc} can be expressed in quotient space with the $L^2$ norm removed from the RHS.}. We will revisit this remark two more times below.
\end{remark}

For the $\rho$ component of the solution, under the zero mean condition in \eqref{Euler:BC}, one has \(\|\rho\|_{L^2(\Omega)}\le C\|\nabla\rho\|_{L^2(\Omega)}\) by the Poincar\'{e} inequality and therefore \be\label{poincare:ineq}\|\rho\|_{H^m(\Omega)}\le C\|\nabla\rho\|_{H^{m-1}(\Omega)}.\ee

Now, define a solution space $\adm^m\subset H^m(\Omega)$ as\[\adm^m:=\left\{U=\bpm\rho\\ u\epm\in H^m(\Omega)\,\Bigr|\,u\ccdot\nm=0\mbox{ on }\pa\Omega\mbox{ and }\int_\Omega\rho\,dx=0\right\}.\]
Then, the above estimates \eqref{elliptic:bc}, \eqref{poincare:ineq} lead to \be\label{elliptic:LK}\|U\|_{H^m(\Omega)}\le C(\|\cL U\|_{H^{m-1}(\Omega)}+\|\cK U\|_{H^{m-1}(\Omega)}+\|u\|_{L^2(\Omega)})\mmbox{for}U\in\adm^m.\ee
 
Next, we move on to Helmholtz decomposition. Define $\cP$ as the $L^2$ projection onto the $L^2$ closure of $\kernel\cL\cap\adm^m$ and define $\cQ$ as its orthogonal complement,\be\label{def:proj}\cP:=L^2\!\!-\!\!\mbox{Proj}\{\overline{\kernel\cL\cap\adm^m}^{L^2}\},\qquad\cQ:={\mathcal I}-\cP.\ee
Note that $\kernel\cL\cap\adm^m=\left\{\bpm0\\u\epm\in H^m\,\Bigr|\,u\ccdot\nm\bc=0,\;\nc u=0\right\}$. 

We present two ways to characterize $\cP,\,\cQ$. The first way is to express them using an elliptic PDE,
\be\label{def:cQ}\mbox{for any }U=\bpm\rho\\ u\epm\in H^m,\;\;\cQ U=\bpm\rho\\\nabla\phi\epm\ee
$\mbox{ with }\phi\mbox{ solving }$\be\label{eq:phi}\left\{\begin{split}\Delta\phi&=\nc u,\\\nabla\phi\ccdot\nm\bigr|_{\pa\Omega}&=u\ccdot\nm\bigr|_{\pa\Omega}.\end{split}\right.\ee
\begin{proposition}Operator $\cQ$ given in \eqref{def:cQ}, \eqref{eq:phi} exists uniquely and amounts to ${\mathcal I}-\cP$ when restricted on $H^m(\Omega) $.
\end{proposition}
\begin{proof}It is sufficient to study the well-posedness of \eqref{eq:phi}. Define a smooth function \[\phi_1(x):=\chi\bigr(\mbox{dist}(x,\,\pa\Omega)\bigr)u(x)\cn\mbox{dist}(x,\,\pa\Omega)\mmbox{ for }x\in\overline{\Omega}\]where $\chi(\ccdot)\in\cC^\infty$ s.t. with a suitable constant $c$, $\chi(z)=1$ if $|z|\le c$ and $\chi(z)=0$ if $|z|>2c$.  It is easy to verify that $\nabla\phi_1\ccdot\nm\bc=u\ccdot\nm\bc$ since $\nabla\mbox{dist}(x,\,\pa\Omega)\bc=-\nm$. Then, one uses the Fredholm alternative to show that \[\left\{\begin{split}\Delta\phi_2&=\nc u-\Delta\phi_1\\\nabla\phi_2\ccdot\nm\bigr|_{\pa\Omega}&=0\end{split}\right.\]admits a $H^{m+1}$ solution $\phi_2$ that is unique module a constant. In fact, this is done by showing $\displaystyle\int_\Omega(\nc u-\Delta\phi_1)\,dx=\int_{\pa\Omega}(u\ccdot\nm-\nabla\phi_1\ccdot\nm)\,ds=0$. Therefore, $\nabla\phi=\nabla(\phi_1+\phi_2)$ in \eqref{def:cQ}, \eqref{eq:phi} exists uniquely in the space of $H^m$.

Next, we verify $\cQ=\cI-\cP$.     By definition,\[\cL(\cI-\cQ) U =\cL \bpm 0\\ u-\nabla\phi\epm =\bpm 0\\ \nc u-\Delta\phi\epm=0,\]\[\mmbox{and}\nabla(u-\nabla\phi)\ccdot\nm\bigr|_{\pa\Omega}=0 .\]Therefore, \[(\cI-\cQ)U\in\kernel\cL\cap\adm^m.\] It remains to show $\cQ U$ is orthogonal to  $\image\cP$. Take a testing function $U'=\bpm\rho'\\ u'\epm\in\kernel \cL\cap\adm^m$ that satisfies $\rho'=0$ and $ \nc u'=0$ with $u'\ccdot\nm\bigr|_{\pa\Omega}=0$. Combining this information, we calculate\[\int_\Omega U'\ccdot\cQ U\,dx=\int_\Omega\rho'\rho+u'\ccdot\nabla\phi\,dx=\int_\Omega\nc(u'\phi)\,dx=\int_{\pa\Omega}\nm\ccdot u'\phi\,ds=0.\]Thus, $\cQ U$ is orthogonal to  ${\kernel\cL\cap\adm^m}$ and therefore its $L^2$ closure.\end{proof}

Now that $\cP$, $\cQ$ are well-defined, for simplicity, we will use $U^P$ for $\cP U$ and $U^Q$ for $\cQ U$ whenever it is not ambiguous. We will also use lower case $u^P$, $u^Q$ for the associated velocity components. The density component of $\cP U$ is always zero.

The second way of characterizing $\cP$ essentially relies on the skew-self-adjointness of $\cL$ and will be very useful in studying nonlinear resonance of acoustic waves.

\begin{proposition}\label{prop:dual}In addition to the definition of $\cP$,   a duality relation also holds true, \be\label{kernel:cP}\kernel\cP\Bigr|_{H^m} =\image\cL\Bigr|_{\adm^{m+1}}.\ee
\begin{proof}Since $\cP$ is a projection, its kernel identifies with $\image(\cI-\cP)$. So it suffices to show\[\image\cQ\Bigr|_{H^m}=\image\cL\Bigr|_{\adm^{m+1}}\]

 Take any $U^Q=\bpm\rho\\u^Q\epm\in H^m$. By \eqref{def:cQ}, \eqref{eq:phi},  $u^Q=\nabla\phi$ for some $\phi\in H^{m+1}$ with $\int_\Omega\phi=0$. Also,  since $\int_\Omega\rho=0$, the Poisson's equation $\Delta \psi=\rho$ with $\nabla\psi\ccdot\nm\bc=0$  admits (at least) a solution in $H^{m+2}$. Thus, we set $U'=\bpm\phi\\\nabla\psi\epm\in \adm^{m+1}$ and clearly $U^Q=\cL U'$. Therefore, LHS\eqref{kernel:cP}$\subset$RHS\eqref{kernel:cP}.

Assume $U=\bpm\rho\\u\epm=\cL U'$ for some $U'=\bpm\rho'\\u'\epm\in \adm^{m+1}$. Since $u=\nabla\rho'$, by \eqref{def:cQ}, \eqref{eq:phi} and its unique solvability, we have $u^Q=u$. Thus, $U=U^Q\in\image\cQ$. Therefore, RHS\eqref{kernel:cP}$\subset$LHS\eqref{kernel:cP}.
\end{proof}
\end{proposition}

\begin{remark}The above duality relation reminds us of Linear Algebra if $\cL$, $\cK$ were (skew)-symmetric matrices and $\cP$, $\cQ$ were matrices representing orthogonal projection onto their null spaces. In fact, if $\Omega$ is contractible, one can use the Poincar\'{e} Lemma to show that $\kernel\cP=\image\cL=\kernel \cK$ in proper regularity spaces. In general, however, only  $ \kernel\cP=\image\cL\subset\kernel\cK$ is true. Take for example $u$ given in \eqref{nontrivial}. It is a nonzero element of both $\kernel\cK$ and $\kernel\cL=\image\cP$ and therefore not of $\kernel\cP$.

On the other hand, regardless of domain topology, both $\image\cL$ and $\kernel\cK$ are invariant spaces of the barotropic compressible Euler equations. In other words, if the initial velocity is a potential flow (resp. curl-free), then it stays so at later times for as long as the classical solution exists. \end{remark}

Now, we claim that $\cP$, $\cQ$ are both bounded operators in $H^m(\Omega)$. Indeed, by definition \(\cL\cP=0\) and thus $\cL\cQ U=\cL U$. Also, by Proposition \ref{prop:dual}, $\image\cQ=\image\cL\subset\kernel\cK$ and thus $\cK\cQ U=0$. Then, elliptic estimate \eqref{elliptic:LK} implies
\be\label{est:cQ:L2}\|\cQ U\|_{H^m}\le C(\|\cL U\|_{H^{m-1}}+\|\cQ U\|_{L^2})\le C\|U\|_{H^m}\ee\be\label{est:cP}\mmbox{ and likewise }\|\cP U\|_{H^{m}}\le C\|U\|_{H^m}.\ee

Even more can be said about $\cQ$.
\begin{proposition}\label{prop:cQ}For any $U^Q\in\image\cQ\cap\adm^m$ ($m\ge1$),\[\|U^Q\|_{H^m}\le C\|\cL U^Q\|_{H^{m-1}}.\]\end{proposition}
\begin{proof}Inequality \eqref{est:cQ:L2} being established, it remains to estimate $\|\cQ U\|_{L^2}$.

By Proposition \ref{prop:dual}, $u^Q=\nabla\phi$ (with $\phi$ chosen to have zero mean). With the assumption $u^Q\ccdot\nm\bc=0$, it follows
\[\begin{split}\int_\Omega |u^Q|^2=&\int_\Omega u^Q\cn\phi=\int_\Omega\nc u^Q\phi\\\le&\|\nc u^Q\|_{L^2}\|\phi\|_{L^2}\le C\|\nc u^Q\|_{L^2}\|\nabla\phi\|_{L^2}.\end{split}\]
Since $u^Q=\nabla\phi$, we deduce $\|u^Q\|_{L^2}\le C\|\nabla\ccdot u^Q\|_{L^2}$ which obviously is bounded by $C\|\cL U^Q\|_{H^{m-1}}$.\end{proof}

A useful application of the above propositions is the weak* convergence result of Theorem \ref{thm:main}.
\begin{corollary}\label{thm:weak}
Let  $U=\bpm\rho\\u\epm$ solves the compressible Euler equations \eqref{Euler}, \eqref{Euler:BC} with sufficient regularity. Then, there exists  constant C depending on the size of $U$ s.t. for any   testing function $U'=\bpm\rho'\\u'\epm\in\cC^1([0,T]\times\overline{\Omega})$, \be\label{c:weak}\left|\int_0^T\!\!\int_\Omega U^Q\ccdot U'\,dxdt\right|\le C{\ep}(\|U'\|_{L^\infty_tL^2_x}+\|\pt U'\|_{L^2_{t,x}}).\ee
\end{corollary}
\begin{proof}Since we always have $\int_\Omega U^Q\ccdot\cP U'=0$, it is enough to consider $U'=\cQ U'$ . Then, by Proposition \ref{prop:dual}, there exists $U''\in\cC^1_{t,x}\cap\adm^2$ s.t. $U'=\cL U''$. Since $\cL U''=\cL\cQ U''$, choose $U''\in\image \cQ$ as well. By Proposition \ref{prop:cQ},
\be\label{est:U2}\|U''\|_{H^1_x}\le C\|U'\|_{L^2_x},\qquad\|\pt U''\|_{H^1_x}\le C\|\pt U'\|_{L^2_x}.\ee

Now, by the skew-self-adjointness of $\cL$ w.r.t. to the solid-wall boundary condition,
\[\begin{split}\int_0^T\!\!\int_\Omega U^Q\ccdot U'\,dxdt=&\int_0^T\!\!\int_\Omega \cL U^Q\ccdot U''\,dxdt=\int_0^T\!\!\int_\Omega \cL U\ccdot U''\,dxdt.\end{split}\]
Then, use the compressible Euler equations \eqref{Euler} to continue
\[\begin{split}...=&-\ep\int_0^T\!\!\int_\Omega (\pt U+\mbox{bounded terms}) \ccdot U''\,dxdt\\
=&-\ep\int_\Omega U\ccdot U''\,dx\Bigr|_0^T+\ep \int_0^T\!\!\int_\Omega\left( U\ccdot\pt U''+\mbox{bounded terms}\ccdot U''\right)\,dxdt.\end{split}\]

Plug in \eqref{est:U2} and the rest of the estimate follows.
\end{proof}

This weak* convergence result can be interpretated in terms of observations in practice. If one chooses $U'$ in \eqref{c:weak} as an integral kernel\[U'=\phi(t-t_0,x-x_0)\quad\mbox{ with }\quad\int\phi\,dxdt=1,\]
and uses it to represent some kind of physical measurement of the velocity and density fields in the fluid, then \eqref{c:weak} confirms that the presence of the fast, oscillatory component in the resulting measurement is of order $O(\ep)$. This is one of the reasons why (almost) incompressible fluid motions are often observed in practice.

 \section{Estimates on Nonlinear Interaction and Strong Convergence}\label{sec:nonlinear}
Rewrite \eqref{Euler} in terms of $U=\bpm\rho\\u\epm$,
\bse\be\label{Euler:sys}\pt U+\cN(U,\nabla U;\ep)=-{1\over\ep}\cL[U],\quad u\ccdot\nm\bc=0,\quad\int_\Omega\rho=0\ee
where the nonlinear term
\be\label{def:cN}\mbox{and }\quad \cN(U_1,\nabla U_2;\ep):=\bpm u_1\cn\rho_2+\rho_1\nc u_2\\ u_1\cn u_2+g(\rho_1;\ep)\nabla\rho_2\epm,\ee\be\label{def:g}\mbox{with }\quad g(\rho;\ep):=\left({p'(1+\ep\rho)\over1+\ep\rho}-1\right){1\over\ep}.\ee\ese
In this section, we always assume the solution to the above system exists and satisfies $U\in\adm^m$ for some $m>\mD/2+4$. It then follows that $\pt U\in\adm^{m-1}$, $\int U\,dt\in\adm^m$, etc.

\begin{theorem}\label{thm:strong}Consider the $\mD$-dimensional barotropic compressible Euler equations \eqref{Euler:sys}. Assume a solution exists classically $U(t,x)\in \cap_{j=0}^m\cC^j([0,T];H^{m-j}(\Omega))$  with $m>\mD/2+4$. Also assume a non-vacuum condition $\|\ep\rho\|_{\cC([0,T];H^{m})}\le1/2$. Then, there exists an incompressible flow $\vV=\bpm0\\\vv\epm$ s.t. for $t\in[0,T]$\[\|\cP\vU-\vV\|_{m-3}\le C\ep F( \max_t\|\vU\|_{m}),\]with projection $\cP$ defined in \eqref{def:proj} and $F(\ccdot)$ some polynomial.

In particular, $\vv$ can be chosen as the unique $H^m$ solution to the incompressible Euler equations
\be\label{inc:Euler}\left\{\begin{split}\pt \vv+\vv\cn \vv+\nabla \overline{q}&=0,\\\nc \vv&=0,\end{split}\right.\ee\[\vv\ccdot\nm\bc=0,\quad \vv(0,\ccdot)=u_0^P.\]
In other words, $\vV$ solves\be\label{inc:sys}\pt\vV+\cP \cN(\vV,\nabla\vV;0)=0.\ee
\end{theorem}
\begin{remark}Here, we avoid using the vorticity formulation for the incompressible Euler equations since vorticity and divergence alone do not necessarily determine a unique velocity field. Consult the example in \eqref{nontrivial}.
\end{remark}

The key to proving this theorem is contained in two lemmas. They have been used in \cite{BC:HYP08} in the torus domain for the rotating shallow water equations. We now migrate them to bounded domains using the properly defined projection operators $\cP$, $\cQ$.

Apply   $\cP$ to \eqref{Euler:sys}. By Proposition \ref{prop:dual}, $\cP\cL=0$. Thus,
\begin{align*} 
\pt\vUn+ \cP \cN(\vU,\nabla\vU;\ep)&=0.
\end{align*}Although $\cN$ defined in \eqref{def:cN} is not entirely bilinear, we take advantage of the facts $p'(1)=1$ and  $\rho^P=0$ to derive \be\label{vUn:sys}\begin{split}
\pt\vUn+ &\cP \cN(\vU^P,\nabla\vU^P;0)+\cP \cN(\vU^Q,\nabla\vU^Q;\ep)\\
+&\cP \cN(\vU^Q,\nabla\vU^P;0)+\cP \cN(\vU^P,\nabla\vU^Q;0)=0
\end{split}\ee
In comparison with  \eqref{inc:sys}, there are two types of nonlinear interactions to be studied,\[\begin{split}\mbox{``fast-fast''}&\qquad\cP \cN(\vU^Q,\vU^Q;\ep),\\
\mbox{``fast-slow''}&\qquad\cP \cN(\vU^Q,\vU^P;0)+\cP \cN(\vU^P,\vU^Q;0).\end{split}\] 

Intuitively, one expects the nonlinearity of $\cN$ to generate resonance   from all kinds of interaction among $\vUn$ and $\vUa$. However, the following lemma \emph{excludes} the contribution of ``fast-fast'' interaction from $\image\cP$.

\begin{lemma}\label{ff:lemma}For any $U^Q\in\image \cQ$ with sufficient regularity,
\[\cP \cN(\vUa,\nabla\vUa;\ep)=0.\]
\end{lemma}
\begin{proof}The density component of $\image\cP$ is always zero.  By Proposition \ref{prop:dual},  it suffices to show that \[u^Q=\nabla \phi\quad\mbox{implies}\quad u^Q\cn u^Q+g(\rho;\ep)\nabla\rho\mbox{ is also a gradient.}\]
This is true due to the Calculus identity\[(\nabla)\phi\cn(\nabla\phi)={1\over2}\nabla|\nabla\phi|^2,\]
and the fact $g(\rho;\ep)$ is a function of $\rho$ only.
\end{proof}
 
Next, move on to ``fast-slow'' interaction: $\cP \cN(\vU^Q,\vU^P;0)+\cP \cN(\vU^P,\vU^Q;0)$. Although it does not have the same cancellation property as ``fast-fast'' interaction, the next lemma reveals that, upon \emph{averaging in time}, it is of $O(\ep)$. In other words, ``fast-slow'' interaction is asymptotically negligible when averaged in time.
\begin{lemma}\label{sf:lemma}For any $U=U^Q+U^P\in\adm^m$ with $m>\mD/2+3$ that solves \eqref{Euler:sys} and satisfies the non-vacuum condition $\|\ep\rho\|_{\cC([0,T];H^{m})}\le1/2$,\[\begin{split}&\left\|\int_0^T\cP \cN(\vUn,\nabla\vUa;0)+\cP \cN(\vUa,\nabla\vUn;0)\,dt\right\|_{H^{m-2}} \\
&\quad\le C\displaystyle{\max_t\|\vU\|^2_{H^m}\left(1+T\max_t\|\vU\|_{H^m}\right)^2}\end{split}\]\end{lemma}
\begin{proof}Since the definition of $\cP$ is independent of the density component, it suffices to estimate\[\int_0^T\cP\bpm 0\\u^P\cn u^Q+u^Q\cn u^P\epm\,dt.\]By \eqref{est:cP},  $\cP$ can be dropped as well.
Now, we demonstrate the estimate on $u^Q\cn u^P$, the other part being very similar.

The idea is to treat the slow component $u^P$ like a constant and focus on averaging the fast oscillatory part $u^Q$ in time. Integrating by parts is therefore in order. But first, a quick calculation reveals that
\begin{align*}
&\int_0^t\cL\vUa\,ds= \int_0^t\cL U\,ds&\text{by }\cL\cP=0\\=&-\ep \int_0^t\pt\vU+\cN(\vU,\nabla\vU;\ep)\,ds&\text{by \eqref{Euler:sys}}\\ =& -\ep\left(U\bigr|_0^t+\int_0^t\cN(\vU,\nabla\vU;\ep)\,ds\right)&\end{align*}
Since Proposition \ref{prop:cQ} yields $\|U^Q\|_{H^m}\le C\|\cL U^Q\|_{H^{m-1}}$ and therefore $\|\int_0^tU^Q\|_{H^m}\le C\|\int_0^t\cL U^Q\|_{H^{m-1}}$, it follows\[\left\|\int_0^tU^Q\right\|_{H^m}\le C\ep\left\|U\Bigr|_0^t+\int_0^t\cN(\vU,\nabla\vU;\ep)\,ds\right\|_{H^{m-1}}\]
By \eqref{cal:ineq}, the RHS is bounded by \[...\le C\ep\left(\max_t\|\vU\|_{m-1}+t\max_t\|\vU\|_{H^m}(\|\vU\|_{H^m}
+\|g(\rho;\ep)\|_{H^{m}})\right).\]
The definition of $g$ in \eqref{def:g}, together with $p'(1)=1$, indicates that its dependence on $\ep$ is essentially via positive powers of $\ep\rho$. The non-vacuum condition $\|\ep\rho\|_{\cC([0,T];H^{m})}\le1/2$ then ensures that $\|g(\rho;\ep)\|_{H^{m-1}}\le C\|\rho\|_{H^{m-1}}$. Therefore, it has been established that
\[\left\|\int_0^tU^Q\right\|_{H^m}\le C\ep \left(\max_t\|\vU\|_{H^m}+t\max_t\|\vU\|_{H^m}^2\right),\]
and, in particular, \[\hat{u}^Q:=\int_0^tu^Q\;\;\mbox{ is bounded in }H^m(\Omega)\mbox{ norm by the RHS above}.\]

Then, perform integrating by parts in time as follows,
\[\int_0^T u^Q\cn u^P\,dt=\hat{u}^Q\cn u^P\Bigr|_0^T-\int_0^T \hat{u}^Q\cn\pt u^P\,dt,\]and arrive at
\[\begin{split}\left\|\int_0^T u^Q\cn u^P\,dt\right\|_{H^{m-2}}\le &C\ep\left(\max_t\|\vU\|_{H^m}+T\max_t\|\vU\|^2_{H^m}\right)\times\\&\left(\max_t\|\vU^P\|_{H^{m}}+T\max_t\|\pt\vUn\|_{H^{m-1}}\right).\end{split}\]
The last term of $\|\pt\vUn\|_{H^{m-1}}$ are estimated using \eqref{vUn:sys} and the fact that $\cN$ depends on $\ep$ via positive powers of $\ep\rho$.
\end{proof}

To close this section, we prove Theorem \ref{thm:strong}.
\begin{proof}[Proof of Theorem \ref{thm:strong}]
Lemma \ref{ff:lemma} and \ref{sf:lemma} together with the integral version of \eqref{vUn:sys} lead to,
\begin{align}\label{r:def} \mbox{ for }-\eps(t,\ccdot):=&u^P\Bigr|_0^t+\int_0^t\tilde{\cP}(u^P\cn u^P)\,ds,\\
\label{est:r}\left\|\eps\right\|_{H^{m-2}(\Omega)}\le &C\ep\displaystyle{\max_t\|\vU\|^2_{H^m}\left(1+T\max_t\|\vU\|_{H^m}\right)^2}.\end{align}
Here and below, $\tilde{\cP}$ is merely the restriction of $\cP$ on the velocity component.
In differential form, $\eps$ satisfies
\be\label{r:pt}-\pt \eps=\pt u^P+\tilde{\cP}(u^P\cn u^P),\quad \eps(0,\ccdot)=0.\ee

Then, subtract \eqref{r:pt} from the incompressible Euler equations \eqref{inc:Euler} and use the fact that $\cI-\tilde{\cP}$ always yields a potential flow to obtain an equation for $e:=\vv-u^P$,
\[\begin{split}\pt \eps=&\pt e+ (e+u^P)\cn(e+u^P)-u^P\cn u^P+\nabla q,\end{split}\]
where $q$ is determined from the incompressibility condition $\nc \eps=\nc e=0$.

Directly applying the energy method to estimate $e$ in the above equation won't work since $\pt \eps$ is of order $O(1)$ by \eqref{r:pt}. To this end, we ``hide'' $\pt \eps$ and look at $e_1 :=e-\eps$ which satisfies
\be\label{e1:pt}\begin{split}-\pt e_1(t)=&(e_1+\eps+u^P)\cn(e_1+\eps+u^P)- u^P\cn u^P+\nabla q\\
=&e_1\cn e_1+e_1\cn f+f\cn e_1 +f^\ep+\nabla q.\end{split}\ee
Here, $f:=\eps+u^P$, $f^\ep:=(\eps+u^P)\cn(\eps+u^P)-u^P\cn u^P$. By \eqref{est:r}, they satisfy estimates \[\|f\|_{H^{m-2}}\le F(\max_t\|U\|_{H^m},T),\qquad\|f^\ep\|_{H^{m-3}}\le \ep F(\max_t\|U\|_{H^m},T)\]for some polynomial $F$.

To apply the energy method to the above system, note that $\eps$, $e$ and $e_1$ all belong to $H^{m-2}(\Omega)$ and satisfy the incompressibility condition and solid-wall boundary condition. Thus, by elliptic estimate \eqref{elliptic:bc},
\[\|e_1\|_{H^{m-3}}\le C(\|e_1\|_{L^2}+\|\nt e_1\|_{H^{m-4}}).\]

It remains to establish upper bounds for $\pt\|e_1\|_{L^2}$ and $\pt\|\nt e_1\|_{H^{m-4}}$ in terms of $\|U\|_{H^m}$. For simplicity, we skip this part, only alluding to the facts that \eqref{e1:pt} is a symmetric Hyperbolic PDE system with dissipative boundary condition and is equipped with a vorticity equation for $\nt e_1$. Also, one can take spatial derivatives on \eqref{e1:pt} up to $(m-3)$-th order and still retain $H^1(\Omega)$ regularity for $\pa_x^{m-3}e_1,\,\pa_x^{m-3}f$ and therefore $H^{1/2}(\pa\Omega)$ regularity for any boundary integrals involved in the process.
 
Finally, we obtain estimates for $\vv-u^P=e_1+\eps$ as desired.
\end{proof}
\section{A Priori Energy Estimates}\label{sec:a:priori}
Recall the equation of state: pressure=$p(1+\ep\rho)$  with $p(1)=0$, $p'(1)=1$ and introduce a new unknown $r:={p(1+\ep\rho)\over \ep}$ s.t.
\be\label{sys:V:ep}\left\{\begin{split}{1\over p^{-1}(\ep r)p'(p^{-1}(\ep r))}\left(r_t+\vu\cn r\right)+{\nc\vu\over\ep}&=0,\\
p^{-1}(\ep r)\left(\vu_t+\vu\cn\vu\right)+{\nabla r\over\ep}&=0.\end{split}\right.\ee
Then, rescale this system by replacing\[u\to\ep u,\;\; r\to\ep r,\;\; t\to{1\over\ep} t\]and arrive at
\[\left\{\begin{split}{1\over p^{-1}(  r)p'(p^{-1}(  r))}\left(r_t+\vu\cn r\right)+{\nc\vu }&=0,\\
p^{-1}( r)\left(\vu_t+\vu\cn\vu\right)+{\nabla r }&=0.\end{split}\right.\]
Finally, rewrite it as a symmetric hyperbolic PDE system for the rescaled variable $V:=\bpm r\\u\epm$, \be\label{sys:V}A_0({V})\pt V+\sum_{j=1}^{\mD}A_j(V)\pa_{x_j} V+{\cL}[V]=0\ee with the obvious definitions of $A_0,\,A_j$. In particular, $A_0(0)=I_{\mD+1}$. 

This rescaled system will be the main subject of this section for it offers the convenience of free of any $\ep$ terms.

Define an energy norm for scalar or vectorial function $f(t,x)$ at any given $t$,
\be\label{def:en}\en{f}_m:=\sqrt{\sum_{|\alpha|+|\beta|\le m}\int_\Omega|\pt^\alpha\px^\beta f(t,x)|^2\,dx}.\ee
Also, introduce some shorthand notations for simplicity.
Let $B_0(t,x),\,B_1(t,x),\,...\,B_{\mD}(t,x)$ denote symmetric-matrix-valued functions of dimension $(\mD+1)\times(\mD+1)$. Define $\vB:=(B_0,....,B_{\mD})$ and $div_{t,x}\vB:=\pa_tB_0+\sum_{j=1}^{\mD}\pa_{x_j}B_j$. Define the boundary matrix using the outward normal $\nm=(\nm_1,...,\nm_\mD)$,
\[B_\nm:=\sum_{j=1}^{\mD}\nm_jB_j \mmbox{on} \pa{\Omega}.\]

The main theorem of this section is regarding \emph{a priori} estimates of linear systems equipped with certain structural features. Estimates for nonlinear systems then follow naturally.

\begin{theorem}\label{thm:linear}Consider the linear symmetric hyperbolic system, \begin{align}\label{lin:V}B_0\pt V+\sum_{j=1}^{\mD}B_j\pa_{x_j} V+{\cL}[V]&=F,\\ 
\label{lin:V:bc}u\ccdot\nm\bc&=0,\end{align}
with possible dissipation $B_\nm\bc\ge0$ and all the  eigenvalues of $B_0$ located in $[1/2,\,2]$. Also, impose a structural assumption that there exists a ``vorticity'' operator $\cK$ of first order differentiation s.t. applying $\cK$ to this system results in a ``vorticity equation'' (at least for smooth solutions)\be\label{structure}B_0\pt\cK[V]+\sum_{j=1}^{\mD}B_j\pa_{x_j}\cK[ V]=F+F_1\ee
where $F_1$ is a sum of products  of entries from $\pa_{x}\vB$ with entries from $\pa_{t,x}V$.

 Then, for any $V,\,\vB \in \sum_{j=0}^m\cC^j([0,T];\,H^{m-j}(\Omega))$ with integer $m>{\mD/2}+1$ satisfying \eqref{lin:V}, \eqref{lin:V:bc}, there exists a universal constant $C=C(m,\Omega)$ s.t.\be\label{en:ineq}\en{V}_m(t)\le C\Bigl(\en{V_0}_m+\en{F}_m +\cE(B) \en{V}_m +\int_0^t\cE(B) \en{V}_m +\en{F}_m\,dt'\Bigr)\ee
where\[\cE(B):=\en{\pa_{t,x}B_0}_{m-1}+\sum_{j=1}^{\mD}\en{B_j}_m.\] \end{theorem}
\begin{remark}
The compressible Euler equations naturally satisfy the structural assumption due to the vorticity equation \[\pt\omega+u\cn\omega+\mbox{lower order terms}=0.\]
\end{remark}
This theorem is an immediate consequence of the following two lemmas. Lemma \ref{lemma:K:pt} uses energy methods and a careful mollification strategy to estimate the space-time norms of the vorticity $\cK V$ and the $L^2$ norms of $\pt^k V$. All boundary integrals in this calculation either vanish or have the correct signs since: 1. the vorticity equation does not have an $\cL$ term; 2. time differentiation of $V$ retains the boundary condition, at least up to $(m-1)$-th order. Lemma \ref{lemma:L} uses elliptic estimates to bound the rest part of $\en{V}_m$ in terms of the already established bounds.

To accommodate the lack of boundary regularity for $\pt^m V\in L^2(\Omega)$,  a mollification process is used in the interior of $\Omega$. We note by passing that it is proved in Rauch \cite{Rauch:char}: $V$ being a weak $L^2$ solution to certain type of hyperbolic PDE system implies the trace $V\ccdot B_{\nm}V$ exists in a very weak sense. For the mere sake of \emph{a priori} estimates, nevertheless, we provide a self-contained argument with somewhat loosened assumptions. To this end, 
pick a $\cC^\infty$ function $\psi(\ccdot)$ with $supp \psi=(-1,1)$, $0\le \psi(\ccdot)\le1$  and $\int_\mR \psi(z)\,dz=1$. Then, define a $\cC^\infty$ mollifier in  the $x$ variable $\psi^x_\de(x):=\prod_{j=1}^{\mD}{1\over\de}\psi\left({x_j\over\de}\right)$ and another one in the $t$ variable $\psi^t_\de(t):={1\over\de}\psi\left({t\over\de}\right)$. 

Also, define $\Omega_\de:=\{x\in\Omega\bigr|dist(x,\pa\Omega)\ge\de\}$. 
\begin{proposition}\label{prop:chi}Let $\chi^t_\de(t,x):={\mathbf 1}_{\tiny{[\de,T-\de]\times\Omega}}$ and $\chi^x_\de(t,x):={\mathbf 1}_{\tiny{[0,T]\times\Omega_\de}}$. Then, for any $f(t,x)\in\cC([0,T];L^2(\Omega))$, as $\de\to0^+$, \[\chi^t_\de\ccdot(\psi_\de^t*f),\qquad \chi^x_\de\ccdot( \psi^x_\de*f)\]both converge to $f(t,x)$    strongly in $L^2(\Omega)$ for fixed $t\in(0,T)$ and  strongly in $L^2([0,T]\times\Omega)$. Moreover, for any $\cC^1$ function $g(t,x)$, the commutators
\[\chi^t_\de\ccdot\pt\left( g\psi^t_\de*f - \psi^t_\de*(gf)\right),\quad\chi^x_\de\ccdot\px\left( g\psi^x_\de*f - \psi^x_\de*(gf)\right)\]converge to 0  weakly in $L^2([0,T]\times\Omega)$.
\end{proposition}
The proof is classic. We remark that the weak convergence is a direct consequence of the strong convergence since one can always choose a testing function in $C^1_0([0,T]\times\Omega)$ to allow integration by parts. This suffices to validate weak convergence since  $C^1_0$ is dense in the $L^2$ topology.


\begin{lemma}\label{lemma:K:pt}Under the same assumptions as Theorem \ref{thm:linear}, \be\label{omega:pt1} \en{ \cK V}_{(m-1,B_0)}\Bigr|_{t_1}^{t_2}+\en{\pt^mV}_{(0,B_0)}\Bigr|_{t_1}^{t_2} \le C\left(\int_{t_1}^{t_2}\en{\pa_{t,x} \vB}_{m-1}\en{V}_m+\en{F}_m\,dt\right),\ee
for $0\le t_1<t_2\le T$. On the LHS, the subscript $B_0$ is understood as a weight that modifies the vector norm $|\cdot|$ in the definition \eqref{def:en}.
\end{lemma}
\begin{proof}By continuity in time, we only need to consider $0<t_1<t_2<T$. Also, it is easy to check that for any continuous and positive-valued function $f(t)$, $g(t)$, $h(t)$,\[f\bigr|_{t_1}^{t_2} \le\int_{t_1}^{t_2}g(t)f(t)+h(t)\,dt\mmbox{iff}f^2\bigr|_{t_1}^{t_2}\le2\int_{t_1}^{t_2}g(t)f^2(t)+h(t)f(t)\,dt.\]Thus, it suffices to prove \be\label{omega:pt} \en{ \cK V}^2_{(m-1,B_0)}\Bigr|_{t_1}^{t_2}+\en{\pt^mV}^2_{(0,B_0)}\Bigr|_{t_1}^{t_2} \le C\left(\int_{t_1}^{t_2}\en{\pa_{t,x} B}_{m-1}\en{V}^2_m+\en{F}_m\en{V}_m\,dt\right).\ee

To estimate $\cK V$, first apply $\pa^\alpha_{t,x}$ with $|\alpha|\le m-1$ to \eqref{lin:V}  and obtain a system for $V^\alpha:=\pa_{t,x}^\alpha V$   well defined in $\cC([0,T]; L^2(\Omega))$,
\be\label{Va}B_0\pt V^\alpha+\sum_{j=1}^{\mD}B_j\pa_{x_j} V^\alpha+\cL[\Va]=F^\alpha\ee
where\[F^\alpha:=\pa^\alpha F-\sum_{j=0}^{\mD}\sum_{|\beta|=1}^{|\alpha|}c_\beta\pa^\beta B_j\pa^{\alpha-\beta}\pa_{x_j}V\]
with $c_\beta=0$   if $\alpha-\beta$ contains negative index. Next,  convolve this $L^2$ system with $\psi^x_\de$ in $x$ variable and apply $\cK$ to obtain,
\be\label{omega:a}B_0\pt  w +\sum_{j=1}^{\mD} B_j\pa_{x_j} w =G_1+G_2\ee
where \[\begin{split} w :=&\cK\psi_\de^x*\Va\\ G_1:=&\cK\psi_\de^x*F^\alpha\\
G_2:=&\left[B_0\pa_t+\sum_{j=1}^{\mD} B_j\pa_{x_j},\,\cK\psi_\de^x*\right]\Va,\end{split}\]
for $(t,x)\in[0,T]\times\Omega_\de$.
Here, we used the facts that $\cK\cL\equiv 0$ and everything is sufficiently smooth.

Now, take  $L^2([t_1,t_2]\times\Omega_\de)$ inner product of \eqref{omega:a} with $ w $, apply the divergence theorem in $t,x$ and use the symmetry of $B_0,...,B_{\mD}$ to calculate,\[\begin{split}&\int_{[t_1,t_2]\times\Omega_\de} w \ccdot(G_1+G_2)=\int_{[t_1,t_2]\times\Omega_\de} w \ccdot B_0\pt  w +\sum_{j=1}^{\mD} w \ccdot B_j\pa_{x_j} w \\= -&\int_{[t_1,t_2]\times\Omega_\de}\Bigl( w \ccdot B_0\pt  w+\sum_{j=1}^{\mD} w \ccdot B_j\pa_{x_j} w + w \ccdot (div_{t,x}  \vB)  w 
\Bigr)+\int_{\Omega_\de} w \ccdot B_0 w\Bigr|_{t_1}^{t_2} +\int_{[t_1,t_2]\times\pa\Omega_\de} w \ccdot B_\nm w \end{split}\]
It follows that
\be\label{div:thm}\int_{\Omega_\de}w\ccdot B_0w\Bigr|_{t_1}^{t_2}=\int_{[t_1,t_2]\times\Omega_\de}\Bigl( 2w\ccdot(G_1+G_2)+w\ccdot (div_{t,x}\vB)w\Bigr) -\int_{[t_1,t_2]\times\pa\Omega_\de} w \ccdot B_\nm w\ee
Here, outward normal $\nm$ is extended smoothly from $\pa\Omega$ to $\overline{\Omega}$ s.t. it is also the outward normal of $\pa\Omega_\de$ for  $\de<1/$curvature of $\pa\Omega$ and vanishes faraway from $\pa\Omega$.

Apply Proposition \ref{prop:chi} on \eqref{div:thm} to pass the strong limit  as $\de\to0^+$ on both sides. More precisely, up to a multiplication of $\chi_\de^x$, \[w(t,\ccdot)\to \pa^\alpha(\cK V(t,\ccdot))\mmbox{strongly in}L^2(\Omega)\mmbox{for fixed}t\]\[w\to \pa^\alpha(\cK V),\qquad G_1\to\cK F^\alpha\mmbox{strongly in}L^2([0,T]\times\Omega)\] 
and therefore the LHS of \eqref{div:thm} and the integral of $w\ccdot G_1$ and $w\ccdot div_{t,x}Bw$ on the RHS   converge to quantities that are bounded by corresponding terms of \eqref{omega:pt}.

Now, we are left with   the integral of $w\ccdot G_2$ and the boundary integral. Subtract and add $\cK(B_0\pa_t+\sum_{j=1}^{\mD} B_j\pa_{x_j})\psi_\de^x*\Va$ in between the two commuting terms of $G_2$,
\[\begin{split}G_2=&(B_0\pa_t+\sum_{j=1}^{\mD} B_j\pa_{x_j})\cK\psi_\de^x*\Va-\cK\psi_\de^x*(B_0\pa_t+\sum_{j=1}^{\mD} B_j\pa_{x_j})\Va\\
=&(B_0\pa_t+\sum_{j=1}^{\mD} B_j\pa_{x_j})\cK\psi_\de^x*\Va-\cK(B_0\pa_t+\sum_{j=1}^{\mD} B_j\pa_{x_j})\psi_\de^x*\Va\\&+\cK\Bigl((B_0\pa_t+\sum_{j=1}^{\mD} B_j\pa_{x_j})\psi^x_\de*\Va-\psi^x_\de*(B_0\pa_t+\sum_{j=1}^{\mD} B_j\pa_{x_j})\Va\Bigr)\\
=:&I_1-I_2+I_2-I_3.
\end{split}\]
By the structural assumption \eqref{structure}, $I_1-I_2$ consists of a sum of products of entries from $\pa_x\vB$ with entries from $\pa_{t,x}\psi_\de^x*\Va$ and therefore is bounded as desired. The $I_2-I_3$ part, by  Proposition \ref{prop:chi}, converges weakly to 0 in $L^2([0,T]\times\Omega)$. We then have shown $\displaystyle\limsup_{\de\to0^+}\int w\ccdot G_2\,dxdt$ is bounded by the RHS of \eqref{omega:pt}.

 And lastly, for the boundary integral $-\int_{[t_1,t_2]\times\pa\Omega_\de} w \ccdot B_\nm w$, since $\nm$ is chosen to be the outward normal of $\pa\Omega_\de$ for $0\le\de<1/$curvature$(\pa\Omega)$, there exists a straight path $X_{\theta}:=x_0+\theta\,\de\,\nm(x_0)$ for $0\le\theta\le1$ that connects any point $x_0=X_0\in\pa\Omega_\de$ to the closest neighbor $x_1=X_1\in\pa\Omega$. Some Geometric Analysis reveals that $\nm(x_0)\bigr|_{\pa\Omega_\de}=\nm(x_1)\bc$. It then follows \[B_\nm(x_1)-B_\nm(x_0)=(x_1-x_0)\int_0^1{\pa\over\pa\nm} B_\nm(X(\theta))\,d\theta\le C\de|\pa_x\vB|_\infty .\]Then, together with the assumption $B_\nm\bc\ge0$, it guarantees
\be\label{bc:int}-\int_{[t_1,t_2]\times\pa\Omega_\de} w \ccdot B_\nm w\le C\de\int_{t_1}^{t_2}|\pa_{x}\vB|_{\infty}\|w\|^2_{L^2(\pa\Omega_\de)}.\ee
Furthermore, the Divergence Theorem implies\[\|w\|^2_{L^2(\Omega_\de)}=\int_{\pa\Omega_\de}(|w|^2\nm)\ccdot\nm=\int_{\Omega_\de}\nc(|w|^2\nm)\le C\|w\|_{L^2(\Omega_\de)}\|w\|_{H^1(\Omega_\de)}\]Use the definition of $w$ and the scaling invariance $\|\nabla\psi_\de^x\|_{L^1 }=C/\de$ to continue,
\[...\le  C\|\cK\Va\|_{L^2}(1+\|\nabla\psi^x_\de\|_{L^1 })\|\cK\Va\|_{L^2}\le (C/\de)\en{V}_m^2.\]Combined with \eqref{bc:int}, it shows  the boundary integral is bounded as desired. We have finished showing the estimates on $\cK V$ in \eqref{omega:pt}.

The study of $\pt^{m-1}V$ involves  a very similar strategy. Let $\alpha=(m-1,0,....0)$ in \eqref{Va} and convolve it with ${d\over dt}\psi_\de^t$,
\be\label{pt:a}B_0\pt  \tilde{w} +\sum_{j=1}^{\mD} B_j\pa_{x_j} \tilde{w}+\cL[\tilde{w}] =\tilde{G}_1+\tilde{G}_2\ee
where \[\begin{split} \tilde{w} :=&\pt\psi_\de^t*\pt^{m-1}V=\pt^{m}\psi_\de^t*V\\ \tilde{G}_1:=&\pt\psi_\de^t*F^\alpha\\
\tilde{G}_2:=&\left[B_0\pa_t+\sum_{j=1}^{\mD} B_j\pa_{x_j},\,\pt\psi_\de^t*\right]\pt^{m-1}V.\end{split}\]
This system is well defined for $(t,x)\in[\de,T-\de]\times\Omega$ and, in particular, $\tilde{w}$ retains the same regularity as  $V$. Thus, $\tilde{w}$ is $C^1$ and satisfies the same solid-wall boundary condition\[\tilde{w}\ccdot\nm\bc=0.\]
Now, follow the same procedure that leads to \eqref{div:thm}, knowing that it can be performed all the way up to $[\de,T-\de]\times\overline{\Omega}$ and the $\cL[\tilde{w}]$ term in \eqref{pt:a} disappears as $\int_\Omega \tilde{w}\ccdot\cL[\tilde{w}]dx=0$ due the above boundary condition,\be\label{div:thm:t}\begin{split}\int_{\Omega}\tilde{w}\ccdot B_0\tilde{w}\Bigr|_{t_0}^{t_1}=&\int_{[t_0,t_1]\times\Omega}\left(2\tilde{w}\ccdot(\tilde{G}_1+\tilde{G}_2)+\tilde{w}\ccdot (div_{t,x}\vB)\tilde{w}\right)-\int_{[t_0,t_1]\times\pa\Omega} \tilde{w} \ccdot B_\nm \tilde{w}\\
\le&\int_{[t_0,t_1]\times\Omega}\left(2\tilde{w}\ccdot(\tilde{G}_1+\tilde{G}_2)+\tilde{w}\ccdot (div_{t,x}\vB)\tilde{w}\right)\end{split}\ee
for $\de\le t_0< t_1\le T-\de$.

Fix $t_1,t_2$ and pass the limit as $\de\to0^+$ in \eqref{div:thm:t} just like what we did for \eqref{div:thm}, we prove \eqref{omega:pt} for $0<t_1<t_2<T$.

\end{proof}

\begin{lemma}\label{lemma:L}Under the same assumption as Theorem \ref{thm:linear}\be\label{LV:k}\begin{split}\en{\pt^k V}_{m-k}\le C\Bigl(\en{\cK V}_{m-1}+\en{\pt^mV}_0+&\en{V}_{m-1}+ \en{F}_m\\&+\bigl(\en{\pa_{t}B_0}_{m-1}+\sum_{j=1}^{\mD}\en{B_j}_m\bigr)\en{V}_m\Bigr)\end{split}\ee for any $0\le k\le m$.
\end{lemma}
\begin{proof}We perform induction w.r.t. $k$.

First, when $k=m$, it is trivial.

Assume \eqref{LV:k} is true for $k= n+1$
\be\label{LV:n}\en{\pt^{n+1}V}_{m-n-1}\le\mbox{RHS}\eqref{LV:k}.\ee
Apply $\pt^n$ on \eqref{lin:V},
\be\label{pt:n}B_0\pt^{n+1}V+\cL[\pt^nV]=F'\ee
where\[F':=F-\pt^{n}(B_0\pt V)+B_0\pt^{n+1}V+\pt^{n}\Bigl(\sum_{j=1}^{\mD}B_j\pa_{x_j} V\Bigr).\]
It is easy to check that $\en{F'}_{m-n-1}\le\mbox{RHS}\eqref{LV:k}$ by virtue of \eqref{cal:ineq}. Likewise
\[\en{B_0\pt^{n+1}V}_{m-n-1}\le|B_0|_\infty\en{\pt^{n+1}V}_{m-n-1}+\mbox{RHS}\eqref{LV:k}\]
Due to the inductive assumption \eqref{LV:n} and the assumption $1/2<B_0<2$, this term is also bounded as desired. Plugging all these estimates into \eqref{pt:n} yields
\be\label{LVV}\en{\cL[\pt^nV]}_{m-n-1}\le \mbox{RHS}\eqref{LV:k}.\ee

Finally, note that elliptic estimate \eqref{elliptic:LK} implies,
\[\begin{split}\en{\px(\pt^nV)}_{m-n-1}\le &C(\en{\cL[\pt^nV]}_{m-n-1}+\en{\cK[\pt^n V]}_{m-n-1}+\en{\pt^nV}_{m-n-1}).\end{split}\]
Combined with \eqref{LVV}, it yields the bound on  $\en{\px(\pt^nV)}_{m-n-1}$ as we expect. Since all the highest order derivatives in $\en{\pt^nV}_{m-n}$ contain at least one $x$ derivative unless it's $\pt^mV$, the above estimate sufficiently   leads to \eqref{LV:k} for $k=n$.
\end{proof}

Obviously, Lemma \ref{lemma:K:pt} and \ref{lemma:L} also apply to any lower order derivatives that appear in the definition of $\en{V}_m$. The very lowest norm $\en{V}_0$ is merely $L^2$ and can be bounded using classical approaches. Therefore, combine all these estimates and finish proving Theorem \ref{thm:linear}.

Now, for the nonlinear system \eqref{sys:V}, we set $B=(A_0(V),A_1(V),...,A_{\mD}(V))$ and $F=0$ in  Theorem \ref{thm:linear}. All conditions are satisfied, esp. the structural assumption \eqref{structure} due to the existence of an actual vorticity equation for the compressible Euler equations. We then use the definitions of $\cE$ and $A_j(V)=A_0(V)u_j$  to arrive at: there exists a  polynomial $\pi(\ccdot)$ of degree $m-1$ with positive constant coefficients only depending on $m$ s.t.\[\cE(B)\le \sum_{k=0}^m|A_0^{(k)}(V)|_\infty\en{V}_m\pi(\en{V}_m).\]Since $A_0(0)=I$ and is defined using  the pressure law $p(\ccdot)\in \cC^\infty$ with $p(1)=0$, $p'(1)=1$,  there exist constants $C_1, C_2$ dependent on $p(\ccdot)$ and $m$ s.t. \[\sum_{k=0}^m|A_0^{(k)}(V)|_\infty\le C_1\mmbox{and}1/2\le A_0(V)\le 2\mmbox{if}|V|_{\infty}\le C_2.\]
Plug it into \eqref{en:ineq} to arrive at, for any $\en{V}_m\le C_2$\[
\en{V}_m(t)\le C\Bigl(\en{V_0}_m+C_1\pi(C_2) \en{V}^2_m +\int_0^tC_1\pi(C_2) \en{V}^2_m\,dt'\Bigr)\]
which further implies, \[
\en{V}_m(t)\le 2C \en{V_0}_m+{1\over C_3}\int_0^t  \en{V}^2_m\,dt' \]if $\en{V}_m$ is also bounded by $C_3= 1/(2CC_1\pi(C_2))$,

By a simple comparison argument, one can show that $\en{V}_m(t)\le a(t)$ for $a(t)$ solving
\[a(t)=a(0)+{1\over C_3} \int_0^ta^2(t') \,dt',\quad a(0)=(2C+1)\en{V_0}_m\]
as long as $\en{V}_m(t)\le\min\{ C_2,\,C_3\}$. Since this integral equation of $a(t)$ is equivalent to a Riccati type ODE, we solve to get $a(t)={1\over {1\over a(0)}-{t\over C_3}}$. Thus,\[\en{V}_m(t)\le{1\over{1\over(2C+1)\en{V_0}_m}-{t\over C_3}}\qquad\mbox{as long as the RHS}\le \min\{ C_2,\,C_3\}\]

Finally, since all the time derivatives in the definition of $\en{V}_m$ can be expressed in terms of spatial derivatives via system \eqref{sys:V}, one has\[\en{V_0}_m\le C_4\|V_0\|_{H^m}\mmbox{if}\|V_0\|_{H^m}\le C_5.\]
\begin{theorem}\label{thm:uniform}Let $m>\mD/2+1$. There exists constant $C^*$, $C^{**}$, $C^{***}$ only dependent on $m,\Omega$ and the pressure law s.t  the rescaled compressible Euler equations \eqref{sys:V} with initial data $V_0$ satisfying the compatibility condition admits a unique solution in the class of \[\sum_{j=0}^m\cC^j\Bigl(\Bigl[0,{C^*\over\|V_0\|_{H^m}}\Bigr];\,H^{m-j}(\Omega)\Bigr)\mmbox{if}  \|V_0\|_{H^m}\le C^{**}.\]The solution is uniformly bounded $\|V\|_{H^{m}(\Omega)}\le C^{***}\|V_0\|_{H^m}$ on this finite time interval.

Consequently, the fast oscillatory system \eqref{sys:V:ep} with initial data $(r_0,u_0)$ satisfying the compatibility condition admits a unique solution in the class of \[\sum_{j=0}^m\cC^j\Bigl(\Bigl[0,{C^*\over\|(r_0,u_0)\|_{H^m}}\Bigr];\,H^{m-j}(\Omega)\Bigr)\mmbox{if}\|(r_0,u_0)\|_{H^m}\le C^{**}/\ep.\]The solution is uniformly bounded $\|(r,u)\|_{H^{m}(\Omega)}\le C^{***}\|(r_0,u_0)\|_{H^m}$ on this finite time interval. \end{theorem}
Here, we assumed the existence theory is automatically valid for the compressible Euler system. For   linear Hyperbolic PDE systems with characteristic boundary, the existence is studied in Lax \& Phillips \cite{LaxPh} in terms of strong $L^2$ solutions and in Rauch \cite{Rauch:char} for $H^m$ regularity. The nonlinear case is studied in Schochet \cite{Sch:Euler,Sch:general}. Our Theorem  \ref{thm:linear}, \ref{thm:uniform} and their proofs provide clear evidence that the continuation principle can be performed on a uniform time interval independent of $\ep$. In particular, when a $\de\nm\ccdot\nabla u$ term is added to convert the original singular boundary matrix to a nonsingular one, the uniform estimate in Theorem \ref{thm:linear} remains valid because: 1. boundary matrix is still dissipative; 2. a vorticity equation still exists. 

We also note by passing that a crucial ``maximal positivity'' condition on the boundary matrix and boundary condition is needed for all these existence theories
to work (cf. Rauch \cite{Rauch:char} and Schochet \cite{Sch:general}), but the system  dealt with here is a canonical case satisfying this condition for sufficiently small $\ep$.

\section{Extension to the Rotating Shallow Water Equations}\label{sec:RSW}
We  can extend the above framework to the  2D Rotating Shallow Water (RSW) equations in a very natural way. Consider a thin layer of fluid moving under the
gravitational force and Coriolis force. Assume the fluid is vertically homogeneous so the dynamics is essentially in 2D. Let $u=\bpm u_1\\u_2\epm$ be the velocity field and $\rho$ be the height perturbation against a constant background height. Then, we formulate
\begin{align}
\label{RSWrho}\pt\rho+\nc(\rho\vu)+{1\over\ep}\nc\vu&=0,\\
\label{RSWu} \pt\vu+\vu\cn\vu+{1\over\ep}\nabla\rho+ {1\over\ep}\vu^\perp&=0.
\end{align}
Here, $\vu^\perp:=\bpm-u_2\\u_1\epm$ reflects the Coriolis effect due to a rotating frame. Impose boundary conditions as
\[u\ccdot\nm\bc=0,\qquad\int_\Omega\rho=0.\]

The analogue between the above system and the compressible Euler equations is as follows. The pressure law for the RSW equations is $p(\rho)={1\over2}\rho^2$, which is due to the gravitational force. The singular parameter $\ep$ is the 2D Froude number and plays the same role as the Mach number. The elliptic operators, with the same notations as before, are
\[\cL U:=\bpm\nc u\\\nabla\rho+\vu^\perp\epm,\]
\[\cK U:=\pa_x u_2-\pa_y u_1-\rho.\]
The RSW system is endowed with a vorticity equation\be\label{omega:RSW}\pt(\cK U)+\nc(\rho\,\cK U)=0.\ee

It remains to justify the Propositions in Section \ref{sec:ell} regarding elliptic estimates and orthogonal projections. The elliptic estimates part is quite straightforward since the $\cL$, $\cK$ defined here are only lower order perturbation of their counterparts in Section \ref{sec:ell}, and one can always bound norms of $\rho$ by norms of $\nabla\rho$. For the orthogonal projection part, we essentially need to find an elliptic PDE to define $\cP$, $\cQ$ which is as convenient as \eqref{eq:phi}. Indeed, since $(\rho,u)\in\kernel\cL$ iff $u=(\nabla\rho)^\perp$, we characterize $\cP$ in the following way,
\[\cP U=\bpm\phi\\(\nabla\phi)^\perp\epm\mmbox{with}(\Delta-1)\phi=\pa_x u_2-\pa_y u_1-\rho,\quad \phi\bc=0,\] 
and, in a much better form,
\be\label{LS}\cP=\cK^*(\cK\cK^*)^{-1}_{D}\cK\ee
where $\cK^*$ is the formal adjoint of $\cK$ and subscript $D$ indicates the Dirichlet boundary condition associated with the elliptic inverse.

The identity \eqref{LS} suggests a parallel argument in Linear Algebra and, in particular, the method of Least Squares.  Without proofs, we state the following set of duality relations which, unlike the Euler equations, are independent of domain topology,
\[\begin{split}\image\cQ=\kernel\cK=\image\cL^*=&\kernel\cP,\\
\image\cP=\kernel\cL=\image\cK^*=&\kernel\cQ.\end{split}\]
Here, $\cL^*=-\cL$.

The last point we would like to make is the analogue to Lemma \ref{ff:lemma} regarding cancellation of ``fast-fast'' interaction under the projection $\cP$. This is still valid for RSW equations due to a simple observation: the vorticity equation \eqref{omega:RSW} results from applying $\cK$ to \eqref{RSWrho}, \eqref{RSWu} and using $\cK\cL=0$. In other words, the identity
\[\cK\bpm\nc(\rho\vu)\\\vu\cn\vu\epm=\nc(\rho\,\cK U)\]
is true regardless of the equation. Now, let $U=U^Q$ be the fast part of a solution. Then, by the duality relations, $\cK U^Q=0$, and the above identity becomes\[\cK\bpm\nc(\rho^Q\vu^Q)\\\vu^Q\cn\vu^Q\epm\equiv0.\]
Again, by the duality relations, we obtain\[\cP\bpm\nc(\rho^Q\vu^Q)\\\vu^Q\cn\vu^Q\epm\equiv0.\]

\section{Acknowledgements}
The author is indebted to Professor J. Rauch for the critical discussions on initial-boundary problems of hyperbolic PDEs. The author also wishes to thank Professors D. Masmoudi and L. Nirenberg  for their valuable advice. Many ideas in this paper were originated in \cite{BC:HYP08} with the encouragement and support from Professor E. Tadmor. Thank you!

  \end{document}